\documentclass[11pt,draft]{amsart}

\usepackage{a4,cite,amsfonts,amsmath,amsthm,amscd,amssymb,latexsym}

\newtheorem{theorem}{Theorem}[section]

\newtheorem{lemma}[theorem]{Lemma}

\theoremstyle{definition}

\newtheorem{question}[theorem]{Question}

\DeclareMathOperator{\con}{con}

\makeatletter

\renewcommand*\subjclass[2][2010]{\def\@subjclass{#2}\@ifundefined{subjclassname@#1}{\ClassWarning{\@classname}{Unknown edition (#1) of Mathematics Subject Classification; using '2010'.}}{\@xp\let\@xp\subjclassname\csname subjclassname@#1\endcsname}}

\makeatother

\renewcommand{\subjclassname}{\textup{2010} Mathematics Subject Classification}

\begin{document}

\title[Cancellable elements of the lattice of semigroup varieties]{Cancellable elements of the lattice\\
of semigroup varieties}

\thanks{The work is partially supported by Russian Foundation for Basic Research (grant 17-01-00551) and by the Ministry of Education and Science of the Russian Federation (project 1.6018.2017).}

\author{S.\,V.\,Gusev}

\address{Ural Federal University, Institute of Natural Sciences and Mathematics, Lenina 51, 620000 Ekaterinburg, Russia}

\email{\{sergey.gusb,dmitry.skokov,bvernikov\}@gmail.com}

\author{D.\,V.\,Skokov}

\author{B.\,M.\,Vernikov}

\begin{abstract}
We completely determine all commutative semigroup varieties that are cancellable elements of the lattice \textbf{SEM} of all semigroup varieties. In particular, we verify that a commutative semigroup variety is a cancellable element of the lattice \textbf{SEM} if and only if it is a modular element of this lattice.
\end{abstract}



\maketitle

\section{Introduction and summary}
\label{intr}

The collection of all semigroup varieties forms a lattice with respect to class-theoretical inclusion. This lattice is denoted by \textbf{SEM}. The lattice \textbf{SEM} has been intensively studied since the beginning of 1960s. A systematic overview of the material accumulated here is given in the survey~\cite{Shevrin-Vernikov-Volkov-09}. There are a number of article devoted to an examination of special elements of different types in the lattice \textbf{SEM} (see~\cite[Section~14]{Shevrin-Vernikov-Volkov-09} or the recent survey~\cite{Vernikov-15} devoted specially to this subject). The present article continues these investigations.

In the lattice theory, special elements of many different types are considered. We recall definitions of three types of elements that appear below. An element $x$ of a lattice $\langle L;\vee,\wedge\rangle$ is called \emph{neutral} if
$$
(\forall y,z\in L)\quad (x\vee y)\wedge(y\vee z)\wedge(z\vee x)= (x\wedge y)\vee(y\wedge z)\vee(z\wedge x).
$$
It is well known that an element $x$ is neutral if and only if, for all $y,z\in L$, the sublattice of $L$ generated by $x$, $y$ and $z$ is distributive (see~\cite[Theorem~254]{Gratzer-11}). Neutral elements play an important role in the general lattice theory (see~\cite[Section~III.2]{Gratzer-11}, for instance). An element $x \in L$ is called
\begin{align*}
&\text{\emph{modular} if}&&(\forall y,z\in L)\quad y\le z\longrightarrow(x\vee y)\wedge z=(x\wedge z)\vee y,\\
&\text{\emph{cancellable} if}&&(\forall y,z\in L)\quad x\vee y=x\vee z\ \&\ x\wedge y=x\wedge z\longrightarrow y=z.
\end{align*}
It is easy to see that any cancellable element is a modular one. A valuable information about modular and cancellable elements in abstract lattices can be found in~\cite{Seselja-Tepavcevic-01}, for instance.

Modular elements of the lattice \textbf{SEM} were examined in~\cite{Jezek-McKenzie-93,Shaprynskii-12,Vernikov-07}. In particular, commutative semigroup varieties that are modular elements of \textbf{SEM} are completely determined in~\cite[Theorem~3.1]{Vernikov-07}. Here we describe commutative semigroup varieties that are cancellable elements of \textbf{SEM}. In particular, we verify that, for commutative varieties, the properties of being modular and cancellable elements are equivalent.

To formulate the main result of the article, we need some notation. We denote by $F$ the free semigroup over a countably infinite alphabet. As usual, elements of $F$ are called \emph{words}. Words unlike variables are written in bold. Two parts of an identity we connect by the symbol~$\approx$, while the symbol~$=$ denotes the equality relation on $F$. Note that a semigroup $S$ satisfies the identity system $\mathbf wx\approx x\mathbf w\approx\mathbf w$ where the variable $x$ does not occur in the word \textbf w if and only if $S$ contains a zero element~0 and all values of \textbf w in $S$ are equal to~0. We adopt the usual convention of writing $\mathbf w\approx 0$ as a short form of such a system and referring to the expression $\mathbf w\approx 0$ as to a single identity. We denote by \textbf T the trivial semigroup variety and by \textbf{SL} the variety of all semilattices.

The main result of the article is the following

\begin{theorem}
\label{main}
For a commutative semigroup variety $\mathbf V$, the following are equivalent:
\begin{itemize}
\item[\textup{a)}] $\mathbf V$ is a cancellable element of the lattice $\mathbf{SEM}$;
\item[\textup{b)}] $\mathbf V$ is a modular element of the lattice $\mathbf{SEM}$;
\item[\textup{c)}] $\mathbf{V=M\vee N}$ where $\mathbf M$ is one of the varieties $\mathbf T$ or $\mathbf{SL}$, while $\mathbf N$ is a variety satisfying the identities $x^2y\approx 0$ and $xy\approx yx$.
\end{itemize}
\end{theorem}

It can be verified by fairly easy calculations that any proper subvariety of the variety \textbf W given by the identities $x^2y\approx 0$ and $xy\approx yx$ is given within \textbf W either by the identity $x^2\approx 0$ or by the identity $x_1x_2\cdots x_n\approx 0$ for some natural $n$ or by both these identities. Thus, in actual fact, Theorem~\ref{main} gives an exhaustive list of the varieties we consider.

The article consists of three sections. Section~\ref{prel} contains an auxiliary facts, while Section~\ref{proof} is devoted to verification of Theorem~\ref{main}.

\section{Preliminaries}
\label{prel}

\subsection{Preliminaries on lattices}
\label{prel lat}

We start with several observations dealing with cancellable or modular elements in abstract lattices.

\begin{lemma}
\label{join with neutral atom}
Let $L$ be a lattice with~$0$ and $a$ an atom and neutral element of $L$. An element $x \in L$ is cancellable if and only if the element $x \vee a$ is cancellable.
\end{lemma}

\begin{proof}
\emph{Necessity.} Let $x$ be a cancellable element and $y,z\in L$. We need to verify that
$$
y\wedge(x\vee a)=z\wedge(x\vee a)\ \&\ y\vee(x\vee a)=z\vee(x\vee a)\longrightarrow y=z.
$$
If $a\le x$ then this implication is evident because $x\vee a=x$ and $x$ is cancellable. Let now $a\not\le x$. Throughout all the proof we will use the fact that the element $a$ is neutral without explicit references. We can assume without loss of generality that either $a\le y$ and $a\le z$ or $a\nleq y$ and $a\nleq z$ or $a\le y$ but $a\nleq z$.

If $a\le y$ and $a\le z$ then
\begin{align*}
\phantom{\text{and}\rule{1.7cm}{0pt}}(y\wedge x)\vee a&=(y\wedge x)\vee(y\wedge a)=y\wedge(x\vee a)\\
&=z\wedge(x\vee a)=(z\wedge x)\vee(z\wedge a)=(z\wedge x)\vee a
\end{align*}
and
$$
(y\wedge x)\wedge a=y\wedge(x\wedge a)=y\wedge 0=0=z\wedge 0=z\wedge(a\wedge x)=(z\wedge x)\wedge a.
$$
Thus, $(y\wedge x)\vee a=(z\wedge x)\vee a$ and $(y\wedge x)\wedge a=(z\wedge x)\wedge a$. The element $a$ is cancellable because it is neutral. Therefore, $y\wedge x=z\wedge x$. Further,
$$
y\vee x=(y\vee a)\vee x=y\vee(x\vee a)=z\vee(x\vee a)=(z\vee a)\vee x=z\vee x.
$$
Thus, $y\wedge x=z\wedge x$ and $y\vee x=z\vee x$. Since $x$ is cancellable, we have $y=z$.

If $a\nleq y$ and $a\nleq z$ then
\begin{align*}
y\wedge x&=(y\wedge x)\vee 0=(y\wedge x)\vee(y\wedge a)=y\wedge(x\vee a)\\
&=z\wedge(x\vee a)=(z\wedge x)\vee(z\wedge a)=(z\wedge x)\vee 0=z\wedge x.
\end{align*}
Thus, $y\wedge x=z\wedge x$. Further,
$$
(y\vee x)\wedge a=(y\wedge a)\vee(x\wedge a)=0\vee 0=(z\wedge a)\vee(x\wedge a)=(z\vee x)\wedge a.
$$
Thus, $(y\vee x)\wedge a=(z\vee x)\wedge a$. By the hypothesis,
$$
(y\vee x)\vee a=y\vee(x\vee a)=z\vee(x\vee a)=(z\vee x)\vee a.
$$
Since $a$ is neutral and every neutral element is cancellable, we have $y\vee x=z\vee x$. Taking into account that the element $x$ is cancellable, we have that $y=z$.

Finally, if $a\le y$ but $a\nleq z$ then
\begin{align*}
z\wedge x&=(z\wedge x)\vee 0=(z\wedge x)\vee(z\wedge a)=z\wedge(x\vee a)\\
&=y\wedge(x\vee a)=(y\wedge x)\vee(y\wedge a)=(y\wedge x)\vee a.
\end{align*}
Thus, $z\wedge x=(y\wedge x)\vee a$. Then $a\le z\wedge x\le z$, a contradiction.

\medskip

\emph{Sufficiency.} Let $x\vee a$ be a cancellable element and $y,z$ are elements of $L$ with $y\wedge x=z\wedge x$ and $y\vee x=z\vee x$. We have to verify that $y=z$. If $a\le x$ then the desirable conclusion is evident because $x\vee a=x$ and the element $x\vee a$ is cancellable. Let now $a\not\le x$. We note that
$$
y\vee(x\vee a)=(y\vee x)\vee a=(z\vee x)\vee a=z\vee(x\vee a),
$$
i.e., $y\vee(x\vee a)=z\vee(x\vee a)$. Since the element $x\vee a$ is cancellable, it remains to check that $y\wedge(x\vee a)=z\wedge(x\vee a)$. As in the proof of necessity, we can assume without loss of generality that either $a\le y$ and $a\le z$ or $a\nleq y$ and $a\nleq z$ or $a\le y$ but $a\nleq z$.

If $a\le y$ and $a\le z$ then
\begin{align*}
y\wedge(x\vee a)&=(y\vee a)\wedge(x\vee a)=(y\wedge x)\vee a\\
&=(z\wedge x)\vee a=(z\vee a)\wedge(x\vee a)=z\wedge(x\vee a),
\end{align*}
i.e., $y\wedge(x\vee a)=z\wedge(x\vee a)$. If $a\nleq y$ and $a\nleq z$ then
\begin{align*}
y\wedge(x\vee a)&=(y\wedge x)\vee(y\wedge a)=(y\wedge x)\vee 0=y\wedge x\\
&=z\wedge x=(z\wedge x)\vee 0=(z\wedge x)\vee(z\wedge a)=z\wedge(x\vee a),
\end{align*}
i.e., $y\wedge(x\vee a)=z\wedge(x\vee a)$ again. Finally, if $a\le y$ but $a\nleq z$ then
\begin{align*}
a&=a\vee(x\wedge a)=(y\wedge a)\vee(x\wedge a)=(y\vee x)\wedge a\\
&=(z\vee x)\wedge a=(z\wedge a)\vee(x\wedge a)=0\vee 0=0,
\end{align*}
a contradiction.
\end{proof}

\begin{lemma}
\label{over neutral atom}
Let $L$ be a lattice with~$0$, $a$ an atom and neutral element of $L$ and $x\in L$. If, for any $y,z\in L$, the equalities $x\vee(y\vee a)=x\vee(z\vee a)$ and $x\wedge(y\vee a)=x\wedge(z\vee a)$ imply that $y\vee a=z\vee a$ then $x$ is a cancellable element.
\end{lemma}

\begin{proof}
Let $y,z\in L$, $x\vee y=x\vee z$ and $x\wedge y=x\wedge z$. We need to verify that $y=z$. It is evident that
$$
x\vee(y\vee a)=(x\vee y)\vee a=(x\vee z)\vee a=x\vee(z\vee a).
$$
Since the element $a$ is neutral, we have
$$
x\wedge(y\vee a)=(x\wedge y)\vee(x\wedge a)=(x\wedge z)\vee(x\wedge a)=x\wedge(z\vee a).
$$
In view of the hypothesis, we have that $y\vee a=z\vee a$. We can assume without loss of generality that either $y,z\ngeq a$ or $y,z\ge a$ or $y\ge a$ but $z\ngeq a$. If $y,z\ngeq a$ then we apply the fact that $a$ is neutral and have
\begin{align*}
y&=(y\vee a)\wedge y=(z\vee a)\wedge y=(z\wedge y)\vee(a\wedge y)=(z\wedge y)\vee 0\\
&=(z\wedge y)\vee(z\wedge a)=z\wedge(y\vee a)=z\wedge(z\vee a)=z,
\end{align*}
i.e., $y=z$. If $y,z\ge a$ then $y=y\vee a=z\vee a=z$. Finally, let $y\ge a$ and $z\ngeq a$. If $x\ge a$ then $x\wedge y\ge a$ and $x\wedge z\ngeq a$. Then $x\wedge y\ne x\wedge z$, contradicting the choice of $y$ and $z$. Let now $x\ngeq a$. Then $x\wedge a=0$ and $z\wedge a=0$. Since $a$ is neutral, we have that 
$$
(x\vee z)\wedge a=(x\wedge a)\vee(z\wedge a)=0\vee 0=0,
$$
whence $x\vee z\ngeq a$. On the other hand, $x\vee y\ge a$. Therefore, $x\vee y\ne x\vee z$ that contradicts the choice of $y$ and $z$ again.
\end{proof}

\begin{lemma}
\label{modular non-cancellable}
Let $x$ be a modular but not cancellable element of a lattice $L$ and let $y$ and $z$ be different elements of $L$ such that $x\vee y=x\vee z$ and $x\wedge y=x\wedge z$. Then there is an element $x'\in L$ such that $x'\le x$, $x'\vee y=x'\vee z$, $x'\wedge y=x'\wedge z$ and $y\vee z=x'\vee y$.
\end{lemma}

\begin{proof}
Put $x'=x\wedge(y\vee z)$. Clearly, $x'\le x$. Note that
$$
x'\wedge y=x\wedge(y\vee z)\wedge y=x\wedge y=x\wedge z=x\wedge(y\vee z)\wedge z=x'\wedge z.
$$
It remains to verify that $x'\vee y=x'\vee z=y\vee z$. Clearly, $x'\le y\vee z$, whence $y\vee x'\le y\vee z$. Then $x\wedge(y\vee z)=x'\le y\vee x'\le y\vee z$, and therefore,
\begin{equation}
\label{(y vee z)wedge x=(y vee x')wedge x}
(y\vee z)\wedge x=(y\vee x')\wedge x.
\end{equation}
Further, the equality $x\vee y=x\vee z$ implies that $z\le y\vee x$. Since $x'\le x$, we have that
$$
(y\vee z)\vee x=(y\vee x)\vee z=y\vee x=y\vee(x'\vee x)=(y\vee x')\vee x.$$
Thus,
\begin{equation}
\label{(y vee z)vee x=(y vee x')vee x}
(y\vee z)\vee x=(y\vee x')\vee x.
\end{equation}
Combining these observations, we have that
\begin{align*}
y\vee z={}&(x\vee(y\vee z))\wedge(y\vee z)&&\\
={}&(x\vee(y\vee x'))\wedge(y\vee z)&&\text{by }\eqref{(y vee z)vee x=(y vee x')vee x}\\
={}&(x\wedge(y\vee z))\vee(y\vee x')&&\text{because }x\text{ is modular and }y\vee x'\le y\vee z\\
={}&(x\wedge(y\vee x'))\vee(y\vee x')&&\text{by }\eqref{(y vee z)wedge x=(y vee x')wedge x}\\
={}&y\vee x'.
\end{align*}
Thus, we prove that $y\vee z=y\vee x'$. Similar arguments allow us to show that $y\vee z=z\vee x'$. Therefore, $y\vee x'=y\vee z=z\vee x'$.
\end{proof}

\subsection{Preliminaries on semigroup varieties}
\label{prel sem var}

Now we return to semigroup varieties. Let \textbf X be a semigroup variety. If nilpotency index of any nil-semigroup in \textbf X is not exceeded some natural number $n$ and $n$ is the least number with such a property then $n$ is called a \emph{degree} of the variety \textbf X and is denoted by $\deg(\mathbf X)$; otherwise we put $\deg(\mathbf X)=\infty$. For a given word \textbf w, we denote by $\ell(\mathbf w)$ the length of \textbf w, and by $\con(\mathbf w)$ the \emph{content} of \textbf w, i.e., the set of all variables occurring in \textbf w. The equivalence of the claims~a) and~c) of the following lemma is verified in~\cite[Proposition~2.11]{Vernikov-08}, the implication c)~$\longrightarrow$~b) is evident, and the implication b)~$\longrightarrow$~a) follows from~\cite[Lemma~1]{Sapir-Sukhanov-81}.

\begin{lemma}
\label{fin deg}
For a semigroup variety $\mathbf V$, the following are equivalent:
\begin{itemize}
\item[\textup{a)}] $\deg(\mathbf V)\le n$;
\item[\textup{b)}] $\mathbf V$ satisfies an identity of the form $x_1x_2\cdots x_n\approx\mathbf v$ for some word $\mathbf v$ with $\ell(\mathbf v)>n$;
\item[\textup{c)}] $\mathbf V$ satisfies an identity of the form
\begin{equation}
\label{=l}
x_1x_2\cdots x_n \approx x_1x_2\cdots x_{i-1}(x_i\cdots x_j)^\ell x_{j+1}\cdots x_n
\end{equation}
for some $\ell>1$ and $1\le i\le j\le n$.\qed
\end{itemize}
\end{lemma}

The following claim is evident.

\begin{lemma}
\label{deg of meet}
If $\mathbf X$ and $\mathbf Y$ are semigroup varieties then
$$
\deg(\mathbf{X\wedge Y})=\min\{\deg(\mathbf X),\deg(\mathbf Y)\}.\eqno{\qed}
$$
\end{lemma}

\begin{lemma}[{\cite[Lemma~2.13]{Vernikov-08}}]
\label{deg of join}
If $\mathbf X$ is a semigroup variety and $\mathbf Y$ is a nil-variety of semigroups then $\deg(\mathbf{X\vee Y})=\max\{\deg(\mathbf X),\deg(\mathbf Y)\}$.\qed
\end{lemma}

We need the following two well known and easily verified technical remarks about identities of nilsemigroups.

\begin{lemma}
\label{split}
Let $\mathbf V$ be a nil-variety of semigroups.
\begin{itemize}
\item[\textup{(i)}] If the variety $\mathbf V$ satisfies an identity $\mathbf{u\approx v}$ with $\con(\mathbf u)\ne\con(\mathbf v)$ then $\mathbf V$ satisfies also the identity $\mathbf u\approx 0$.
\item[\textup{(ii)}] If the variety $\mathbf V$ satisfies an identity of the form $\mathbf{u\approx vuw}$ where at least one the words $\mathbf v$ and $\mathbf w$ is non-empty then $\mathbf V$ satisfies also the identity $\mathbf u\approx 0$.\qed
\end{itemize}
\end{lemma}

The first statement of the following lemma is generally known (see~\cite[Section~1]{Shevrin-Vernikov-Volkov-09}, for instance). The second claim also is well known and is verified explicitly in~\cite[Proposition~2.4]{Volkov-05} (see also~\cite[Section~14]{Shevrin-Vernikov-Volkov-09}).

\begin{lemma}
\label{SL is neutral atom}
The variety $\mathbf{SL}$ is
\begin{itemize}
\item[\textup{(i)}] an atom of the lattice $\mathbf{SEM}$;
\item[\textup{(ii)}] a neutral element of $\mathbf{SEM}$.\qed
\end{itemize}
\end{lemma}

\section{Proof of the main result}
\label{proof}

In this section we prove Theorem~\ref{main}. The implication a)~$\longrightarrow$~b) is evident, while the equivalence of the claims b) and c) is checked in~\cite[Theorem~3.1]{Vernikov-07}. It remains to prove the implication c)~$\longrightarrow$~a). Lemmas~\ref{join with neutral atom} and~\ref{SL is neutral atom} allow us to assume that $\mathbf V=\mathbf N$. Suppose that \textbf N is non-cancellable element of \textbf{SEM}. Hence there are semigroup varieties \textbf Y and \textbf Z with $\mathbf{N\vee Y}=\mathbf{N\vee Z}$, $\mathbf{N\wedge Y}=\mathbf{N\wedge Z}$ and $\mathbf{Y\ne Z}$.

\begin{lemma}
\label{deg Y=deg Z}
$\deg(\mathbf Y)=\deg(\mathbf Z)$.
\end{lemma}

\begin{proof}
Put $\deg(\mathbf Y)=r$, $\deg(\mathbf Z)=s$ and $\deg(\mathbf N)=t$ (here $r,s,t\in\mathbb N\cup\{\infty\}$). Suppose that $r\ne s$. We can assume without any loss that $r<s$. Then Lemmas~\ref{deg of meet} and~\ref{deg of join} imply that
\begin{itemize}
\item[] if $t\ge s$ then $\deg(\mathbf{N\wedge Y})=r<s=\deg(\mathbf{N\wedge Z})$;
\item[] if $r<t<s$ then $\deg(\mathbf{N\wedge Y})=r<t=\deg(\mathbf{N\wedge Z})$;
\item[] if $t\le r$ then $\deg(\mathbf{N\vee Y})=r<s=\deg(\mathbf{N\vee Z})$.
\end{itemize}
The first and the second cases contradict the equality $\mathbf{N\wedge Y}=\mathbf{N\wedge Z}$, while the third case is impossible because $\mathbf{N\vee Y}=\mathbf{N\vee Z}$.
\end{proof}

Since the claims~b) and~c) of Theorem~\ref{main} are equivalent, \textbf N is a modular element of \textbf{SEM}. In view of Lemma~\ref{modular non-cancellable}, there is a variety $\mathbf N'$ such that
$$
\mathbf{N'\subseteq N},\,\mathbf{N'\vee Y}=\mathbf{N'\vee Z}=\mathbf{Y\vee Z}\text{ and }\mathbf{N'\wedge Y}=\mathbf{N'\wedge Z}.
$$
Being a subvariety of \textbf N, the variety $\mathbf N'$ satisfies the identities $x^2y\approx 0$ and $xy\approx yx$.

Since $\mathbf{Y\ne Z}$, we can assume without loss of generality that there is an identity $\mathbf{u\approx v}$ that holds in \textbf Y but is false in \textbf Z. If this identity is satisfied by the variety $\mathbf N'$ then it holds in $\mathbf{N'\vee Y}=\mathbf{N'\vee Z}$, and therefore, in \textbf Z. Thus, $\mathbf{u\approx v}$ is wrong in $\mathbf N'$. A word \textbf w is called \emph{linear} if any variable occurs in \textbf w at most once. Recall that $\mathbf N'$ satisfies the identities $x^2y\approx 0$ and $xy\approx yx$. Therefore, any non-linear word except $x^2$ equals to~0 in $\mathbf N'$. Thus, we may assume without loss of generality that either $\mathbf u=x^2$ or $\mathbf u=x_1x_2\cdots x_k$ for some $k$. Lemmas~\ref{over neutral atom} and~\ref{SL is neutral atom} allow us to assume that $\mathbf Y,\mathbf{Z\supseteq SL}$. This implies that $\con(\mathbf u)=\con(\mathbf v)$. Combining the observations given above, we have that $\mathbf{u\approx v}$ is either an identity of the form $x^2\approx x^m$ for some $m\ne 2$ or an identity of the form $x_1x_2\cdots x_k\approx\mathbf v$ where $\con(\mathbf v)=\{x_1,x_2,\dots,x_k\}$.

\smallskip

\emph{Case} 1: $\mathbf{u\approx v}$ is an identity of the form $x^2\approx x^m$ for some $m\ne 2$. Suppose at first that $m=1$. This means that \textbf Y is a variety of bands. Then $\mathbf{Z\wedge N'}=\mathbf{Y\wedge N'}=\mathbf T$. If $\mathbf N'=\mathbf T$ then $\mathbf Y=\mathbf{Y\vee N'}=\mathbf{Z\vee N'}=\mathbf Z$, and we are done. Otherwise, $\mathbf N'$ contains the variety \textbf{ZM} of all semigroups with zero multiplication. Since $\mathbf{Z\wedge N'}=\mathbf T$, we have that $\mathbf{Z\nsupseteq ZM}$, whence the variety \textbf Z is completely regular. If \textbf Z contains a non-trivial group variety \textbf G then $\mathbf{G\subseteq Z\vee N'}=\mathbf{Y\vee N'}$. But all groups in $\mathbf{Y\vee N'}$ are trivial because this variety satisfies the identity $x^3\approx x^4$. Thus, \textbf Z is a completely regular variety without non-trivial groups, i.e., a band variety. We see that the identity $\mathbf{u\approx v}$ holds in \textbf Z, a contradiction.

Let now $m>2$. If $\mathbf N'$ satisfies the identity $x^2\approx 0$ then the identity $x^2\approx x^m$ holds in the variety $\mathbf{N'\vee Y}=\mathbf{N'\vee Z}$, and therefore, in \textbf Z. But this contradicts the choice of the identity $\mathbf{u\approx v}$. Thus we can assume that the identity $x^2\approx 0$ is wrong in $\textbf N'$. Recall that a word \textbf w is called an \emph{isoterm for a variety} \textbf V if \textbf V does not satisfy any non-trivial identity of the form $\mathbf{w\approx w}'$. Lemma~\ref{split} implies that the word $x^2$ is an isoterm for the variety $\mathbf N'$. Further, Lemma~\ref{split}(ii) implies that the variety $\mathbf{N'\wedge Z}=\mathbf{N'\wedge Y}$ satisfies the identity $x^2\approx 0$. Therefore, there is a deduction of this identity from identities of the varieties $\mathbf N'$ and \textbf Z. In particular, one of these varieties satisfies a non-trivial identity of the form $x^2\approx\mathbf w$. Since $x^2$ is an isoterm for $\mathbf N'$, this identity holds in \textbf Z. Since $\mathbf{Z\supseteq SL}$, this identity has the form $x^2\approx x^k$ for some $k>2$. Let $m$ be the least number with the property that $x^2\approx x^m$ holds in \textbf Y but does not hold in \textbf Z, while $k$ the least number with the property that $x^2\approx x^k$ holds in \textbf Z.

Suppose that $k<m$. Then $m=k+j$ for some natural $j$. It is clear that the identity $x^{2+j}\approx x^{k+j}=x^m$ holds in $\mathbf N'$. Then this identity is true also in $\mathbf{Z\vee N'}=\mathbf{Y\vee N'}$. Hence $x^{2+j}\approx x^m\approx x^2$ holds in \textbf Y. Since $2+j<m$, this contradicts the choice of $m$.

Finally, let $m<k$. Then $k=m+j$ for some natural $j$. Clearly, the identity $x^{2+j}\approx x^{m+j}=x^k$ holds in $\mathbf N'$. Therefore, this identity holds in $\mathbf{Y\vee N'}=\mathbf{Z\vee N'}$. This means that \textbf Z satisfies the identities $x^{2+j}\approx x^k\approx x^2$. But $2+j<m+j=k$ and we have a contradiction with the choice of $k$.

\smallskip

\emph{Case} 2: $\mathbf{u\approx v}$ is an identity of the form $x_1x_2\cdots x_k\approx\mathbf v$ where $\con(\mathbf v)=\{x_1,x_2,\dots,x_k\}$. Clearly, $\ell(\mathbf v)\ge k$. If $\ell(\mathbf v)=k$ then the identity $\mathbf{u\approx v}$ has the form
$$
x_1x_2\cdots x_k\approx x_{1\pi}x_{2\pi}\cdots x_{k\pi}
$$
where $\pi$ is a non-trivial permutation on the set $\{1,2,\dots, k\}$. This identity holds in $\mathbf N'$ because $\mathbf N'$ is commutative. But this is false. Therefore, $\ell(\mathbf v)>k$. Put $\deg(\mathbf Y)=n$. Then $\deg(\mathbf Z)=\deg(\mathbf Y)=n$ by Lemma~\ref{deg Y=deg Z}. Lemma~\ref{fin deg} implies that $n\le k$. Recall that $\mathbf{Y\vee Z}=\mathbf{N'\vee Y}=\mathbf{N'\vee Z}$. Clearly, $\deg(\mathbf Y\vee\mathbf Z)\ge n$. Suppose at first that $\deg(\mathbf Y\vee\mathbf Z)=n$. Then
$$
\deg(\mathbf N')\le\deg(\mathbf{N'\vee Y})=\deg(\mathbf{Y\vee Z})=n.
$$
Being a nil-variety, $\mathbf N'$ satisfies the identity $x_1x_2\cdots x_n\approx 0$ in this case. Since $\ell(\mathbf v)>k\ge n$, the identity $x_1x_2\cdots x_k\approx\mathbf v$ holds in $\mathbf N'$ as well. This contradicts the choice of the identity $\mathbf{u\approx v}$.

Let now $\deg(\mathbf{Y\vee Z})>n$. Since $\deg(\mathbf Y)=n$, Lemma~\ref{fin deg} implies that \textbf Y satisfies an identity of the form~\eqref{=l} for some $\ell>1$ and $1\le i\le j\le n$. The same lemma implies that this identity is false in $\mathbf{Y\vee Z}$ because $\deg(\mathbf{Y\vee Z})=n$ otherwise. Therefore,~\eqref{=l} is wrong in \textbf Z. Analogously, there are $r>1$ and $1\le i'\le j'\le n$ such that the identity
\begin{equation}
\label{=r}
x_1x_2\cdots x_n \approx x_1x_2\cdots x_{i'-1}(x_{i'}\cdots x_{j'})^rx_{j'+1}\cdots x_n
\end{equation}
holds in \textbf Z but does not hold in \textbf Y. We will assume without any loss that $i\le i'$.

Suppose at first that $j<j'$. Then we substitute $(x_{i'}\cdots x_{j'})^{r-1}x_{j'+1}$ into $x_{j'+1}$ in~\eqref{=l} whenever $j'<n$ or multiply~\eqref{=l} by $(x_{i'}\cdots x_{j'})^{r-1}$ on the right whenever $j'=n$. We obtain the identity
\begin{equation}
\label{r=l[r-1]}
\begin{array}{rl}
&x_1x_2\cdots x_{i'-1}(x_{i'}\cdots x_{j'})^rx_{j'+1}\cdots x_n\\
\approx{}&x_1x_2\cdots x_{i-1}(x_i\cdots x_j)^\ell x_{j+1}\cdots x_{j'}(x_{i'}\cdots x_{j'})^{r-1}x_{j'+1}\cdots x_n.
\end{array}
\end{equation}
Clearly, the identity~\eqref{r=l[r-1]} holds in the variety $\mathbf N'$. Then it satisfies in \textbf Z as well because $\mathbf{N'\vee Y}=\mathbf{N'\vee Z}$. Substitute $x_{i-1}(x_i\cdots x_j)^{\ell-1}$ into $x_{i-1}$ in~\eqref{=r} whenever $i>1$ or multiply~\eqref{=r} by $(x_i\cdots x_j)^{\ell-1}$ on the left whenever $i=1$. As a result, we obtain the identity
\begin{equation}
\label{l=[l-1]r}
\begin{array}{rl}
&x_1x_2\cdots x_{i-1}(x_i\cdots x_j)^\ell x_{j+1}\cdots x_n\\
\approx{}&x_1x_2\cdots x_{i-1}(x_i\cdots x_j)^{\ell-1}x_i\cdots x_{i'-1}(x_{i'}\cdots x_{j'})^rx_{j'+1}\cdots x_n.
\end{array}
\end{equation}
This identity holds in \textbf Z too. Note that the right parts of the identities~\eqref{r=l[r-1]} and~\eqref{l=[l-1]r} coincide. Indeed,
\begin{align*}
&x_1x_2\cdots x_{i-1}(x_i\cdots x_j)^\ell x_{j+1}\cdots x_{j'}(x_{i'}\cdots x_{j'})^{r-1}x_{j'+1}\cdots x_n\\
={}&x_1x_2\cdots x_{i-1}(x_i\cdots x_j)^{\ell-1}x_i\cdots x_jx_{j+1}\cdots x_{j'}(x_{i'}\cdots x_{j'})^{r-1}x_{j'+1}\cdots x_n\\
={}&x_1x_2\cdots x_{i-1}(x_i\cdots x_j)^{\ell-1}x_i\cdots x_{i'}\cdots x_{j'}(x_{i'}\cdots x_{j'})^{r-1}x_{j'+1}\cdots x_n\\
={}&x_1x_2\cdots x_{i-1}(x_i\cdots x_j)^{\ell-1}x_i\cdots x_{i'-1}(x_{i'}\cdots x_{j'})^rx_{j'+1}\cdots x_n.
\end{align*}
Since the variety \textbf Z satisfies the identities~\eqref{=r},~\eqref{r=l[r-1]} and~\eqref{l=[l-1]r}, this variety satisfies also the identity~\eqref{=l}. We have a contradiction.

It remains to consider the case when $j'\le j$. Suppose at first that $i=i'$ and $j=j'$. Substitute $(x_{i'}\dots x_{j'})^rx_{j'+1}$ into $x_{j'+1}$ in~\eqref{=l} whenever $j'<n$ or multiply~\eqref{=l} by $(x_{i'}\cdots x_{j'})^r$ on the right whenever $j'=n$. Then we obtain the identity
\begin{equation}
\label{r=[r+l]}
x_1x_2\cdots x_{i-1}(x_i\cdots x_j)^rx_{j+1}\cdots x_n\approx x_1x_2\cdots x_{i-1}(x_i\cdots x_j)^{r+\ell}x_{j+1}\cdots x_n.
\end{equation}
Clearly, this identity holds in $\mathbf N'$. The equality $\mathbf{Y\vee N'}=\mathbf{Z\vee N'}$ implies that it holds in \textbf Z too. Similar arguments show that \textbf Z satisfies the identity
\begin{equation}
\label{l=[r+l]}
x_1x_2\cdots x_{i-1}(x_i\cdots x_j)^\ell x_{j+1}\cdots x_n\approx x_1x_2\cdots x_{i-1}(x_i\cdots x_j)^{r+\ell}x_{j+1}\cdots x_n.
\end{equation}
Combining the identities~\eqref{=r},~\eqref{r=[r+l]} and~\eqref{l=[r+l]}, we have that \textbf Z satisfies the identity~\eqref{=l}, contradicting with the choice of this identity.

Thus, either $i<i'$ or $j'< j$. Suppose without loss of generality that $i<i'$. Substitute $x_{i'-1}(x_{i'}\cdots x_{j'})^{r-1}$ into $x_{i'-1}$ in~\eqref{=l}. We obtain the identity
\begin{equation}
\label{r=[r]l}
\begin{array}{rl}
&x_1x_2\cdots x_{i'-1}(x_{i'}\cdots x_{j'})^rx_{j'+1}\cdots x_n\\
\approx{}&x_1x_2\cdots x_{i-1}(x_i\cdots x_{i'-1}(x_{i'}\cdots x_{j'})^rx_{j'+1}\cdots x_j)^\ell x_{j+1}\cdots x_n.
\end{array}
\end{equation}
Clearly, the identity~\eqref{r=[r]l} holds in the variety $\mathbf N'$. Besides that, it holds in \textbf Z because $\mathbf{N'\vee Y}=\mathbf{N'\vee Z}$. For an arbitrary word \textbf w, we suppose $\mathbf w^0$ to be the empty word. Let $t>0$ and $s\ge 0$. Now we multiply the identity~\eqref{=r} by the word
$$
(x_i\cdots x_{i'-1}(x_{i'}\cdots x_{j'})^rx_{j'+1}\cdots x_j)^s
$$
on the left whenever $i=1$ or substitute the word
$$
x_{i-1}(x_i\cdots x_{i'-1}(x_{i'}\cdots x_{j'})^rx_{j'+1}\cdots x_j)^s
$$
into $x_{i-1}$ in~\eqref{=r} whenever $i>1$. Besides that, we multiply~\eqref{=r} by the word $(x_i\cdots x_j)^{t-1}$ on the right whenever $j=n$ or substitute the word $(x_i\cdots x_j)^{t-1}x_{j+1}$ into $x_{j+1}$ in~\eqref{=r} whenever $j<n$. Then we obtain the identity{\sloppy

}
\begin{equation}
\label{(r)sl=(r)[s+1][l-1]}
\begin{array}{rl}
&x_1x_2\cdots x_{i-1}(x_i\cdots x_{i'-1}(x_{i'}\cdots x_{j'})^rx_{j'+1}\cdots x_j)^s(x_i\cdots x_j)^tx_{j+1}\cdots x_n\\
\approx{}&x_1x_2\cdots x_{i-1}(x_i\cdots x_{i'-1}(x_{i'}\cdots x_{j'})^rx_{j'+1}\cdots x_j)^{s+1}(x_i\cdots x_j)^{t-1}\cdot\\
&\cdot\,x_{j+1}\cdots x_n.
\end{array}
\end{equation}
Then the variety \textbf Z satisfies the identities
$$
\begin{array}{rcl}
x_1x_2\cdots x_n&\stackrel{\eqref{=r}}\approx&x_1x_2\cdots x_{i'-1}(x_{i'}\cdots x_{j'})^rx_{j'+1}\cdots x_n\\
&\stackrel{\eqref{r=[r]l}}\approx&x_1x_2\cdots x_{i-1}(x_i\cdots x_{i'-1}(x_{i'}\cdots x_{j'})^rx_{j'+1}\cdots x_j)^\ell x_{j+1}\cdots x_n\\
&\stackrel{\eqref{(r)sl=(r)[s+1][l-1]}}\approx&x_1x_2\cdots x_{i-1}(x_i\cdots x_{i'-1}(x_{i'}\cdots x_{j'})^rx_{j'+1}\cdots x_j)^{\ell-1}(x_i\cdots x_j)\cdot\\
&&\cdot\,x_{j+1}\cdots x_n\\
&\stackrel{\eqref{(r)sl=(r)[s+1][l-1]}}\approx&x_1x_2\cdots x_{i-1}(x_i\cdots x_{i'-1}(x_{i'}\cdots x_{j'})^rx_{j'+1}\cdots x_j)^{\ell-2}(x_i\cdots x_j)^2\cdot\\
&&\cdot\,x_{j+1}\cdots x_n\\
&\hdotsfor{2}\\
&\stackrel{\eqref{(r)sl=(r)[s+1][l-1]}}\approx&x_1x_2\cdots x_{i-1}(x_i\cdots x_{i'-1}(x_{i'}\cdots x_{j'})^rx_{j'+1}\cdots x_j)(x_i\cdots x_j)^{\ell-1}\cdot\\
&&\cdot\,x_{j+1}\cdots x_n\\
&\stackrel{\eqref{(r)sl=(r)[s+1][l-1]}}\approx&x_1x_2\cdots x_{i-1}(x_i\cdots x_j)^\ell x_{j+1}\cdots x_n.
\end{array}
$$
Here we write $\mathbf{w\stackrel{\varepsilon}\approx w'}$ in the case when the identity $\mathbf{w\approx w'}$ follows from the identity $\varepsilon$. We use the identity~\eqref{(r)sl=(r)[s+1][l-1]} for the first time with $s=\ell-1$ and $t=1$, for the second time with $s=\ell-2$ and $t=2$, \dots, for the penultimate time with $s=1$ and $t=\ell-1$, finally, for the last time with $s=0$ and $t=\ell$. We prove that the identity~\eqref{=l} holds in \textbf Z, a contradiction. This completes the proof of Theorem~\ref{main}.\qed

\medskip

At the conclusion of the article, we formulate some open questions.

\begin{question}
\label{cancellable = modular?}
Does there exist a semigroup variety that is a modular but not a cancellable element of the lattice \textbf{SEM}?
\end{question}

A semigroup variety is called 0-\emph{reduced} if it may be given by identities of the form $\mathbf w\approx 0$ only. It is known that any 0-reduced semigroup variety is a modular element of the lattice \textbf{SEM}. This fact was noted for the first time in~\cite[Corollary~3]{Vernikov-Volkov-88} and rediscovered (in different terminology) in~\cite[Proposition~1.1]{Jezek-McKenzie-93}. In actual fact, it readily follows from~\cite[Proposition~2.2]{Jezek-81}.

\begin{question}
\label{0-reduced is cancellable?}
Is any 0-reduced semigroup variety a cancellable element of the lattice \textbf{SEM}?
\end{question}

Evidently, the negative answer to Question~\ref{0-reduced is cancellable?} immediately implies the negative answer to Question~\ref{cancellable = modular?}. An affirmative answer to Question~\ref{0-reduced is cancellable?} would also have an interesting corollary. To formulate it, we recall that an element $x$ of a lattice $L$ is called \emph{lower-modular} if
$$
(\forall\,y,z\in L)\quad x\le y\longrightarrow x\vee(y\wedge z)=y\wedge(x\vee z).
$$
Lower-modular elements of the lattice \textbf{SEM} are completely determined in~\cite{Shaprynskii-Vernikov-10}. This result easily implies that if an answer to Question~\ref{0-reduced is cancellable?} is affirmative then every lower-modular element of \textbf{SEM} is cancellable.

\end{document}